\documentclass[11pt]
{amsart}
\usepackage{amssymb,amsmath,amsthm,amsfonts,amsopn,url,color,hyperref,
mathtools,microtype,MnSymbol,dutchcal,
mathrsfs}
\usepackage{scalerel}
\usepackage{stackengine,wasysym}
\usepackage[shortlabels]{enumitem}

\usepackage[normalem]{ulem}
\usepackage[all]{xy}
\input xy
\xyoption{all}
\usepackage{amscd}
\usepackage{soul}

\theoremstyle{plain}
\newtheorem{thm}{Theorem}[section]

\newtheorem{prop}[thm]{Proposition}
\newtheorem{lemma}[thm]{Lemma}
\newtheorem{cor}[thm]{Corollary}

\theoremstyle{definition}
\newtheorem{defn}[thm]{Definition}
\newtheorem*{defn*}{Definition}
\newtheorem*{question*}{Question}
\newtheorem{question}{Question}
\newtheorem{example}[thm]{Example}
\newtheorem*{example*}{Example}
\newtheorem{rem}[thm]{Remark}
\newtheorem*{rem*}{Remark}

\newcommand{\field}[1]{\mathbb{#1}}
\newcommand{\N}{\field{N}}
\newcommand{\Z}{\field{Z}}
\newcommand{\Q}{\field{Q}}
\newcommand{\R}{\field{R}}

\newcommand{\ideal}[1]{\mathfrak{#1}}
\newcommand{\m}{\ideal{m}}

\newcommand{\bfx}{\mathbf{x}}

\newcommand{\ra}{\rightarrow}

\newcommand{\be}{\begin{enumerate}}
\newcommand{\ee}{\end{enumerate}}

\newcommand{\li}
 {\leftfootline}

\newcommand{\onto}{\twoheadrightarrow}

\newcommand{\cA}{\mathcal{A}}

\renewcommand{\phi}{\varphi}

\let\int\relax
\DeclareMathOperator{\int}{i}

\author{Neil Epstein}
\address{Department of Mathematical Sciences \\ George Mason University \\ Fairfax, VA  22030}
\email{nepstei2@gmu.edu}

\title{Egyptian integral domains}

\date{August 11, 2023}
\begin{document}
\maketitle
\begin{abstract}
The notion of an \emph{Egyptian domain} (where the analogue of Egyptian fractions works appropriately), first explored by Guerrieri-Loper-Oman, is extended to the more general notions of \emph{generically} and \emph{locally} Egyptian domains.  Results from the previous paper are extended, as well as reinterpreted in the new expanded context.  It is shown that localizations of polynomial rings are typically Egyptian, that positive semigroup-graded domains are \emph{not} Egyptian, and that affine semigroup rings are not Egyptian unless the semigroup is a group.  However, any finitely generated algebra over a field or $\Z$ is shown to be \emph{locally} Egyptian.  It is further shown that if $R$ is a pullback of $S$, then $R$ is Egyptian if and only if $S$ is.
\end{abstract}

\section{Introduction}

In \cite{GLO-Egypt}, Guerrieri, Loper, and Oman transport the study of Egyptian fractions from the realms of mathematical history and number theory to the study of commutative integral domains. They say an element of the fraction field of an integral domain $D$ is \emph{Egyptian} if it can be written as sums of reciprocals of distinct elements of $D$.  They show that it is equivalent to say that the element can be written as sums of reciprocals of not necessarily distinct elements.  A domain $D$ is called \emph{Egyptian} if every nonzero element of $D$ (or equivalently, every nonzero element of its fraction field) is Egyptian.  The result implicitly relied upon by the ancient Egyptians in the 2\textsuperscript{nd} millenium BCE (and proved by Fibonacci in the middle ages) is that rational numbers between $0$ and $1$ are Egyptian.  In \cite{GLO-Egypt}, the authors develop some theory and give several classes of Egyptian (and non-Egyptian) domains.  In particular, they show that any domain with a nontrivial Jacobson radical (e.g. any local domain other than a field) is Egyptian, meaning that \emph{failure} to be Egyptian is in some sense a global property.  They also show that any overring of an Egyptian domain is Egyptian, and any proper overring of a polynomial ring in one variable over a field is Egyptian.  On the other hand, they show that domains of the form $D[X]$ (with $D$ any domain), and certain ultraproducts, are \emph{not} Egyptian.  More is also done in this paper rich of exposition, results, and an invitation to the interested reader to join the study of Egyptian domains.

In the current paper, I accept their invitation, in that I generalize some of the results in \cite{GLO-Egypt} (and in the process I provide further examples and non-examples of Egyptianness) and introduce the novel concepts 
of \emph{locally} and \emph{generically Egyptian} domains.  I show that a finitely generated $k$-algebra (or $\Z$-algebra) domain is always locally Egyptian, but rarely Egyptian.  I also show that the properties of Egyptian and generically Egyptian descend in pullbacks, which has consequences for integral extensions within a fraction field.  Like \cite{GLO-Egypt}, the current paper closes with some interesting questions.

Let us collect together here most of the main results in \cite{GLO-Egypt} so that we may refer to them later.

\begin{prop}[{\cite[Proposition 1]{GLO-Egypt}}]\label{pr:polynot} For any integral domain $D$, the polynomial ring $D[X]$ is not Egyptian.  Even the element $X$ is not Egyptian.
\end{prop}

\begin{thm}\label{thm:prevEgyptian}
If an integral domain $D$ satisfies any of the following conditions, then it is Egyptian.
\begin{enumerate}
\item\label{it:field} $D$ is a field. \cite[Example 1]{GLO-Egypt}
\item\label{it:Z} $D=\mathbb Z$. {\cite[Theorem 1]{GLO-Egypt}}\footnote{The authors of \cite{GLO-Egypt} attribute this result to Fibonacci's 1202 book \textit{Liber Abaci}, cited from its translation to English in \cite{DuGr-FibE}, though they provide their own proof.  However, reading Fibonacci's work from that source, it is clear that the latter only proved that \emph{proper} fractions are Egyptian.  On the other hand, work in the mid-20th century~\cite{VanVan-rbaseZ} shows that one can represent any positive integer as the sum of reciprocals from an arbitrary arithmetic sequence.  Coupled with Fibonacci's actual result, (\ref{it:Z}) follows.}
\item\label{it:nondistinct} Every nonzero $d\in D$ can be expressed as a sum of (not necessarily distinct) unit fractions. $($from \cite[Theorem 2]{GLO-Egypt}$)$
\item\label{it:overpoly} $D$ is a proper overring of $F[X]$, where $F$ is a field. \cite[Proposition 7]{GLO-Egypt}
\item\label{it:overring} $D$ is an overring of an Egyptian domain. \cite[Proposition 4]{GLO-Egypt}
\item\label{it:alg} $D$ is algebraic over an Egyptian domain. \cite[Proposition 5]{GLO-Egypt}
\item\label{it:Groupring} $D = A[G]$, where $A$ is an Egyptian domain and $G$ is a torsion-free abelian group. \cite[Proposition 3]{GLO-Egypt}
\item\label{it:Jac} The Jacobson radical of $D$ is nonzero. \cite[Example 3]{GLO-Egypt}
\end{enumerate}
\end{thm}

\section{Semigroup algebras and their localizations}
In this section, I generalize both the theorem that proper overrings of $F[X]$ ($F$ a field) are Egyptian, and also the theorem that rings of the form $A[X]$ ($A$ any domain) are \emph{not} Egyptian.  Using some convexity theory, one then obtains a characterization of which affine semigroup rings are Egyptian.
\begin{prop}\label{pr:polyloc}
    Let $D$ be an Egyptian domain. Let $W$ be a multiplicative subset of $D[X]$ such that $W \nsubseteq D$.  Then $R=D[X]_W$ is Egyptian.
\end{prop}

\begin{proof}
By Theorem~\ref{thm:prevEgyptian}(\ref{it:nondistinct}) it is enough to show that each element in a set of generators of $R$ (as a ring) is Egyptian. Accordingly, note that $R$ is generated as a ring by $D\setminus \{0\}$, $W^{-1} = \{\frac 1w \mid w \in W\}$, and the element $X$.  Any element of $D \setminus \{0\}$ is Egyptian in $D$ (and hence in $R$) by assumption.  Any element of $W^{-1}$ is the reciprocal of an element of $D[X]$, hence of $R$, and is thus Egyptian.  As for $X$, choose an element $g \in W$ of positive degree.  If $g$ has no constant term, we have $g=Xh$, so that $X = g/h = \frac 1 {h/g}$, yielding an Egyptian representation of $X$ in $R$ since $h/g \in R \setminus \{0\}$.  Otherwise, let $-d$ be the constant term of $g$.  Then $g+d = Xh$ for some $h \in D[X]$, and since $D$ is Egyptian we have $d = \sum_{i=1}^n \frac 1{c_i}$ for some nonzero elements $c_i$ of $d$.  Thus, \[
X = \frac gh + \sum_{i=1}^n \frac 1 {c_i h} = \frac 1 {h/g} + \sum_{i=1}^n \frac 1 {c_i h},
\]
which shows that $X$ is Egyptian in $R$ since the fraction $h/g$ and the products $c_i h$ all are nonzero elements of $R$.
\end{proof}

\begin{rem}
Since for a field $F$, $F[X]$ is a QR domain, the above proposition is a generalization of Theorem~\ref{thm:prevEgyptian}(\ref{it:overpoly}).  In fact it is a proper generalization, as it follows, for instance, that $\mathbb Z[X, 1/g]$ is Egyptian for any nonconstant polynomial $g \in \mathbb Z[X]$.  One might wish to prove the following more natural-looking generalization: ``Let $D$ be an Egyptian domain, $F$ its fraction field, and let $E$ be a proper overring of $D[X]$ such that $E \cap F = D$. Then $E$ is Egyptian". But that generalization is false.  Let $D = \Z$ and $E = \Z[X/3]$.  Then $E$ is isomorphic to a polynomial ring in one variable over $\Z$, and thus by Proposition~\ref{pr:polynot}, $E$ is not Egyptian.
\end{rem}

For the next result, we need to recall some basics of graded rings.  A nice source for this is \cite[4.A]{BrGu-polybook}.  If $\Gamma$ is a commutative monoid, written additively, then a \emph{$\Gamma$-graded ring} is a ring $R$, along with a collection of additive subgroups $\{R_g \mid g\in \Gamma\}$ of $R$, such that $R = \bigoplus_{g\in \Gamma} R_g$ and such that for any $g,h\in \Gamma$, we have $R_g \cdot R_h \subseteq R_{g+h}$.  For $g\in \Gamma$, an element of $R_g$ is called \emph{homogeneous}, of \emph{degree} g.  We say a $\Gamma$-grading on $R$ is \emph{nontrivial} if $R \neq R_0$.  That is, there is some homogeneous element of nonzero degree.

\begin{prop}\label{pr:nonneg-graded}
Let $G$ be a totally ordered abelian group.  Let $\Gamma$ be the monoid of nonnegative elements (i.e. elements greater than or equal to the identity element $0$) of $G$, and let $D$ be a nontrivially $\Gamma$-graded domain.  Then $D$ is not Egyptian.
\end{prop}

\begin{proof}
Let $x \in D$ be a homogeneous element of degree $d >0$.  Suppose $x$ is Egyptian.  Then there exist nonzero $f_1, \ldots, f_n \in D$ such that \[
x = \frac 1{f_1} + \cdots + \frac 1{f_n}.
\]
For each $1\leq i \leq n$, let $d_i$ be the degree of $f_i$ (i.e. the degree of the top degree homogeneous component, or \emph{leading term}, of $f_i$).  From the displayed equation, we can clear denominators to get \[
xf_1 \cdots f_n = \sum_{i=1}^n \prod_{j\neq i} f_j.
\]
But here we reach a contradiction, as the left hand side has degree equal to $d+d_1 + \cdots +d_n$ \cite[Lemma 4.6]{BrGu-polybook}, which is strictly greater than $d_1 + \cdots + d_n$, whereas each summand of the right hand side has degree less than or equal to $d_1 + \cdots + d_n$ (again by \cite[Lemma 4.6]{BrGu-polybook}).
\end{proof}

For the rest of this section, we need some convexity theory. A good reference for this is \cite[Chapter 1]{Vil-monbook}.  If $K$ is a subfield of $\R$, we set $K_{\geq 0} := \{k \in K \mid k\geq 0\}$.  If $\cA$ is a nonempty subset of $K^n$, where $K$ is an ordered subfield of $\R$, and $U$ is a subset of $K$, then $U \cA := \{u_1 a_1 + \cdot + u_t a_t \mid t \in \N$, $u_j \in U$, $a_j \in \cA\}$.  This covers the symbols $\R\cA$, $\Q_{\geq 0} \cA$, and $\R_{\geq 0} \cA$.  Note that $K \cA$ is a $K$-linear subspace of $\R^n$.

We also need the following lemma, which must be well-known, but for lack of a reference will give a proof.  Note that the finitely generated case appears as \cite[Exercise 2.5]{BrGu-polybook}.

\begin{lemma}\label{lem:groupcone}
Let $\Lambda$ be an additive submonoid of $\Q^n$.  Suppose $\R_{\geq 0} \Lambda$ is a group.  Then $\Lambda$ is a group.
\end{lemma}

\begin{proof}
Let $0 \neq x \in \Lambda$.  We have $-x \in \R_{\geq 0}\Lambda \cap \Q^n$.  Write $-x = \sum_{i=1}^\ell r_i y_i$, with $r_i \in \R_{\geq 0}$ and $y_i \in \Lambda$.  There is a positive integer $d$ such that $dx \in \Z^n$ and $z_i := dy_i \in \Z^n$ for all $1 \leq i \leq \ell$.  Write $\cA = \{z_1, \ldots, z_\ell\}$.  Then $-dx = \sum_{i=1}^\ell r_i z_i \in \R_{\geq 0} \cA \cap \Z^n$.  Then by \cite[Corollary 1.1.27]{Vil-monbook}, $-dx \in \Q_{\geq 0}\cA \cap \Z^n$.
Thus, there exist $q_1, \ldots, q_\ell \in \Q_{\geq 0}$ with $-dx = \sum_{i=1}^\ell q_i z_i$.  The $q_i$ have a positive common denominator, say $m$, so that $q_i = b_i / m$, with $m$ a positive integer and the $b_i$ nonnegative integers.  Then \[
-mdx = \sum_{i=1}^\ell b_i z_i \in \Lambda.
\]
Since also $(md-1)x \in \Lambda$, we have $-x = -mdx + (md-1)x \in \Lambda$.
\end{proof}

\begin{thm}\label{thm:forcepositivegrading}
Let $n$ be a positive integer, let $\Lambda$ be an additive submonoid of $\Q^n$, let $A$ be an integral domain, and let $D = A[\Lambda]$ be the corresponding semigroup ring.  If $\Lambda$ is not a group, $D$ admits a nontrivial $\R_{\geq 0}$-grading. If moreover $\Lambda$ is finitely generated, then $D$ even admits a nontrivial $\N$-grading.
\end{thm}

\begin{proof}
Let $d = \dim \R\Lambda$.  By Lemma~\ref{lem:groupcone}, $C :=\mathbb R_{\geq 0} \Lambda$ is not a group.  Thus, there is some $f\in \R_{\geq 0} \Lambda$ such that $-f \notin \R_{\geq 0} \Lambda$.  We have $f = \beta g$ for some $\beta \in \R_{\geq0}$ and $g\in \Lambda$.  Then  $g\in \Lambda$, but $-g \notin \R_{\geq 0} \Lambda$. 
 Choose a linear automorphism $T: \R^n \ra \R^n$ such that $T(\R\Lambda) = \R^d \times 0^{n-d}$.  Let $\pi: \R^n \onto \R^d$ be the projection onto the first $d$ coordinates.  Set $B := (\pi \circ T)(C)$.  Note that $\pi \circ T$ induces a bijection between $B$ and $C$.  Note also that  $B$ is a convex set, as it is closed under sums and positive scaling.  Moreover, $h:= \pi(T(g))$ is a nonzero element of $B$, and $-h = \pi(T(-g)) \notin B$.  It follows that the origin is a boundary point of $B$, as any ball in $\R^d$ that contains the origin contains  positive multiples of $h$ (which are in $B$) and of $-h$ (which are not in $B$).
 
 Then by the Supporting Hyperplane Theorem \cite[2.5.2]{BoVa-convexbook}, there is some ${\mathbf 0} \neq v\in \R^d$ such that $v \cdot b \leq v \cdot {\mathbf 0} = 0$ for all $b \in B$, where we are taking ordinary dot product in $\R^d$.  Since $v \cdot v > 0$, it follows that $v \notin B$.  If we also had $-v \notin B$, then for any nonzero scalar $c$, we would have $cv \notin B$, contradicting the fact that $\R B = \R^d$.  Thus, $-v \in B$.

 Now define a function $\delta: C \ra \R$ by $\delta(w) = \pi(T(w)) \cdot (-v)$.  By construction, the image of $\delta$ lies in $\R_{\geq 0}$.  Now, let $v' := (v,0,\ldots, 0) \in \R^n$, and let $\mathbf e = T^{-1}(-v')$.  Then $\mathbf e \in C$ and $\delta(\mathbf e) = \pi(T(\mathbf e)) \cdot (-v) = (-v) \cdot (-v) = v\cdot v>0$.  Choose $\alpha \in \R_{> 0}$ such that $\alpha \mathbf e \in \Lambda$; then $\delta(\alpha \mathbf e) = \alpha\delta(\mathbf e) >0$.  For each nonnegative scalar $c$, let $D_c$ be the direct sum of all the free rank 1 $A$-modules of the form $Am$, where $m = \bfx^{\mathbf a} = x_1^{a_1} x_2^{a_2} \cdots x_n^{a_n}$, $(a_1, \ldots, a_n) \in \Lambda$, and $\delta(a_1, \ldots, a_n) = c$.  Then $D_c$ is itself a free $A$-module, and as $A$-modules we have $D = \bigoplus_{c\in \R_{\geq 0}} D_c$. Moreover, if $m$ is a monomial in $D_c$ and $m'$ is a monomial in $D_{c'}$, say $m = \bfx^{\mathbf a}$ and $m' = \bfx^{\mathbf a'}$, then $\delta(\mathbf{a+a'}) = \pi(T(\mathbf{a+a'})) \cdot (-v) = \pi(T(\mathbf a))\cdot (-v) + \pi(T({\mathbf{a'}})) \cdot (-v) = \delta(\mathbf a) + \delta(\mathbf a')=c+c'$, so that $mm' = \bfx^{\mathbf a} \bfx^{\mathbf a'} = \bfx^{\mathbf{a+a'}} \in D_{c+c'}$.  It follows that $D_c D_{c'} = D_{c+c'}$. Thus, $D$ is an $\R_{\geq 0}$-graded domain that contains a homogeneous element $x^{\alpha \mathbf e}$ of positive degree.

Finally, suppose $\Lambda$ is finitely generated, say by $u_1, \ldots, u_t$. Let $d = \dim_\R \R\Lambda = \dim_\Q \Q\Lambda$.  Choose a $\Q$-linear automorphism $S: \Q^n \ra \Q^n$ that induces a $\Q$-isomorphism from $\Q \Lambda$ onto $\Q^d \times 0^{n-d}$.  By scaling, we may replace $S$ with a map $T$ that has the same properties but also the added property that all the $u_i$ are sent to elements of $\Z^d \times 0^{n-d}$. Let $\pi: \Q^n \onto \Q^d$ be the projection onto the first $d$ coordinates.  Then $\pi(T(\Lambda))$ is a finitely generated submonoid of $\Z^d$, such that $\Q_{\geq 0}\pi(T(\Lambda))$ is full-dimensional but does not fill the space.  By \cite[Theorem 1.1.29 and Proposition 1.1.51]{Vil-monbook}, there is some $a\in \Z^d$ such that $x \cdot a \geq 0$ for all $x\in \pi(T(\Lambda))$, and such that there exists some $v\in \pi(T(\Lambda))$ with $v \cdot a >0$.  Define $\delta: \Lambda \rightarrow \N_0$ by $\delta(w) = \pi(T(w)) \cdot a$.  Then by the same argument as above, $\delta$ induces a nontrivial $\N$-grading on $D$, via transforming the exponent vectors of monomials.
\end{proof}

As a corollary: 

\begin{thm}\label{thm:notagroupring}
Let $n$ be a positive integer, let $\Lambda$ be an additive submonoid of $\Q^n$, let $A$ be an integral domain, and let $D = A[\Lambda]$ be the corresponding semigroup ring.  If $\Lambda$ is not a group, then $D$ is not Egyptian.
\end{thm}

\begin{proof}
Combine Proposition~\ref{pr:nonneg-graded} and Theorem~\ref{thm:forcepositivegrading}.
\end{proof}

\begin{rem}
The above is a wide-ranging generalization of Proposition~\ref{pr:polynot}.  It includes the previous result by letting $\Lambda = \N$ and $D=A[X]$.  But it also shows that a wide variety of interesting rings fail to be Egyptian, even rings like $A[X^2, X^3]$ and $A[X^3, X^2 Y, XY^2, Y^3]$, $A[x/y, y/z^3]$, etc.

Note also that taken along with Theorem~\ref{thm:prevEgyptian}(\ref{it:Groupring}), we have a complete classification of which affine semigroup rings (with Egyptian domain base ring) are Egyptian.  It's precisely the ones that are group rings.
\end{rem}

\section{Locally and generically Egyptian domains}
In order to handle the fact that not every integral domain is Egyptian, it makes sense to consider some weakenings of the notion.  The notion of \emph{locally Egyptian} is like finding a covering by Egyptian open patches, and \emph{generically Egyptian} is like finding a dense open subset that is Egyptian.  Among other things, it is proved here that local and generic Egyptianness are well behaved when extending to finitely generated algebras, provided there is no kernel in the structure map.  This stands in contrast to the Egyptian property itself.

\noindent \textbf{Notation:} For a domain $D$ and an element $0 \neq d\in D$, we will use $D[1/d]$ and $D_d$ interchangeably.
\begin{defn}
Let $D$ be an integral domain.  We say that $D$ is \emph{locally Egyptian} if there is some finite system $f_1, \ldots, f_n$ of generators of the unit ideal of $D$ such that $D[1/f_j]$ is Egyptian for $1 \leq j \leq n$.
\end{defn}

\begin{rem}
It might at first seem more natural to define ``locally Egyptian'' to mean that $D_\m$ is Egyptian for all maximal ideals $\m$.  But this is automatically true for \emph{all} domains (by Theorem~\ref{thm:prevEgyptian}(\ref{it:Jac})), and hence useless.  The definition here has the sense of ``can cover with open affine subsets of the right type''.  One could specify by calling the property ``\emph{affine}-locally" (rather than stalk-locally) Egyptian.

Note also that the definition of ``locally Egyptian'' would be the same if we just called for any system of generators for the unit ideal (rather than restricting to a finite one), as any system of generators of the unit ideal contains a finite subsystem.
\end{rem}

\begin{example}[Oman]\label{ex:FXlocally}
If $F$ is a field and $D=F[X]$, then even though $D$ is not Egyptian (see Proposition~\ref{pr:polynot}), it \emph{is} locally Egyptian, as both $D[1/X]$ and $D[1/(1-X)]$ are Egyptian by Theorem~\ref{thm:prevEgyptian}(\ref{it:overpoly}), and the pair $\{X, 1-X\}$ clearly generate the unit ideal of $D$.
\end{example}

For the next Theorem, we provide a more general definition:

\begin{defn}
An integral domain $D$ is \emph{generically Egyptian} if there is some nonzero $f\in D$ such that $D[1/f]$ is Egyptian.
\end{defn}

\begin{rem}
 The following implications hold: \[
\text{Egyptian} \implies \text{locally Egyptian} \implies \text{generically Egyptian}
\]
\end{rem}

Like the Egyptian property itself (see Theorem~\ref{thm:prevEgyptian}(\ref{it:overring})), both ``locally Egyptian'' and ``generically Egyptian'' pass to overrings.

\begin{lemma}\label{lem:overring}
Let $R$ be an integral domain and $S$ an overring.  If $R$ is generically (resp. locally) Egyptian, then so is $S$.
\end{lemma}

\begin{proof}
Suppose $R$ is generically Egyptian.  Let $0\neq r \in R$ such that $R_r$ is Egyptian.  Then since $S_r$ is an overring of $R_r$, it follows from Theorem~\ref{thm:prevEgyptian}(\ref{it:overring}) that $S_r$ is Egyptian.  Hence $S$ is generically Egyptian.

Suppose $R$ is locally Egyptian.  Let $f_1, \ldots, f_n$ be nonzero elements of $R$ that generate the unit ideal of $R$ such that each $R_{f_i}$ is Egyptian.  Then the $f_i$ also generate the unit ideal of $S$, and since $S_{f_i}$ is an overring of $R_{f_i}$ for each $1 \leq i \leq n$, we again have by Theorem~\ref{thm:prevEgyptian}(\ref{it:overring}) that $S_{f_i}$ is Egyptian.  Hence $S$ is locally Egyptian.
\end{proof}

\begin{example}
[Guerrieri]\label{ex:infinitepoly} Not every integral domain is generically Egyptian.  For example, let $D = k[X_1, X_2, X_3, \ldots]$ be a polynomial ring in countably many variables over a field $k$.  Then for any $0 \neq f \in D$, there is some $n$ such that $f \in k[X_1,\ldots, X_n]$.  Hence, $D_f = k[X_1, \ldots, X_n]_f[X_{n+1}, X_{n+2}, \ldots] = (k[X_1, \ldots, X_n, X_{n+2}, \ldots]_f)[X_{n+1}]$ is a polynomial ring over an integral domain, and hence not Egyptian due to Proposition~\ref{pr:polynot}.  Since $D$ is not generically Egyptian, it is not locally Egyptian either.
\end{example}

\begin{thm}\label{thm:fgalgebra}
Let $D$ be a generically (resp. locally) Egyptian domain, with fraction field $K$.  Let $L$ be an extension field of $K$, and let $R$ be a ring with $D \subseteq R \subseteq L$ such that $R$ is finitely generated as a $D$-algebra.  Then $R$ is generically (resp. locally) Egyptian.
\end{thm}

For this, we start with the following lemmas.

\begin{lemma}\label{lem:fg}
Let $D$, $R$ be as in Theorem~\ref{thm:fgalgebra} and assume $D$ is Egyptian.  Let $u_1, \ldots, u_t$ be a generating set of $R$ as a $D$-algebra, and let $u = \prod_{i=1}^tu_i$.  Then $R_u$ is Egyptian.
\end{lemma}

\begin{proof}
First note that $R_u = R[1/u_1,\ldots, 1/u_t]$.  Let $0 \neq \alpha \in R_u$.  Then $\alpha = r/u^t$ for some $0 \neq r \in R$ and $t\in \N_0$.  Then there are monomials $m_j$ in the $u_i$ and $d_j \in D$ such that $r = d_1m_1 + \cdots + d_n m_n$.  But each $m_j$ is invertible in $R_u$.  Thus, $\alpha = \frac r {u^t} = \sum_{j=1}^n \frac {d_j}{u^tm_j^{-1}}$.  Since each $d_j$ is Egyptian in $D$ (and hence in $R$ and $R_u$), the result follows.
\end{proof}

\begin{proof}[Proof of Theorem~\ref{thm:fgalgebra}]
Let $u_1, \ldots, u_t$ be a generating set of $R$ as a $D$-algebra, and let $u = \prod_{i=1}^tu_i$. Suppose that $D$ is generically Egyptian. Let $d \in D$ such that $D[1/d]$ is Egyptian.  Then by Lemma~\ref{lem:fg}, since the $u_i$ also generate $R_d$ over $D_d$, we have that $(R_d)_u = R_{du}$ is Egyptian.

Now suppose $D$ is locally Egyptian.  Choose $d_1, \ldots, d_k$ that generate the unit ideal of $D$ such that $D[d_i^{-1}]$ is Egyptian for each $1\leq i \leq k$.  For each subset $S \subseteq [t] := \{i\in \N \mid 1 \leq i \leq t\}$, set $C_S := \{u_i \mid i \in [t] \setminus S\} \cup \{1-u_i \mid i \in S\}$, and $u_S := \prod\{v \mid v\in C_S\}$.  Then for each $S \subseteq [t]$, we have $D[C_S] = R$.  Thus by Lemma~\ref{lem:fg}, for each pair $(i,S)$, we have that $(R_{d_i})_{u_S} = R_{u_S d_i}$ is Egyptian. Moreover, $\sum_{S \subseteq [t]} u_S = 1$ (as one sees from the expansion of $\prod_{i=1}^t (u_i + v_i)$ as a sum of monomials and specializing $v_i \mapsto 1-u_i$), and there exist $a_i \in D$ such that $\sum_{i=1}^k a_i d_i =1$.  Thus, \[
\sum_{i=1}^k \sum_{S \subseteq [t]} a_i u_S d_i = \sum_{i=1}^k a_i \left(\sum_{S \subseteq [t]} u_S \right) d_i = \sum_{i=1}^k a_i \cdot 1 \cdot d_i = \sum_{i=1}^k a_i d_i =1.
\]
Therefore, the set $\{u_S d_i \mid 1\leq i \leq k,$ $S \subseteq [t]\}$ generates the unit ideal of $R$.  Since each $R_{u_S d_i}$ is Egyptian, it follows that $R$ is locally Egyptian.
\end{proof}

\begin{rem}
There is a stark contrast between Theorem~\ref{thm:notagroupring} and Theorem~\ref{thm:fgalgebra}.  Combining these theorems, we see that an affine semigroup ring over a field or $\Z$ is always \emph{locally} Egyptian, but only \emph{Egyptian} in the case one already knows from \cite{GLO-Egypt}, wherein the semigroup in question is a group.
\end{rem}

\begin{cor}
Let $A$ be either a field or a domain that is an integral extension of
$\Z$ (e.g. the ring of integers of an algebraic number field).  Then any integral domain $R$ that is essentially of finite type over $A$ is locally Egyptian.
\end{cor}

\begin{proof}
Recall that a ring $S$ is \emph{essentially of finite type} over a ring $B$ if it is a localization of a finitely generated $B$-algebra.  But since any localization of an Egyptian domain is Egyptian by Theorem~\ref{thm:prevEgyptian}(\ref{it:overring}), we may assume $R$ is a finitely generated $A$-algebra.

First suppose $A$ is a field.  Then the hypotheses of Theorem~\ref{thm:fgalgebra} are satisfied with $D=K=A$, with $D$ Egyptian by Theorem~\ref{thm:prevEgyptian}(\ref{it:field}), hence locally Egyptian.

It remains to handle the case where $A$ is an integral extension of $\Z$. Let $\mu: A \ra R$ be the structure map. By \cite[Theorem 48]{Kap-CR}, $\dim A =1$. Hence, $\ker \mu$ is either the zero ideal or a maximal ideal.  If $\ker \mu$ is a maximal ideal $\m$, then the hypotheses of Theorem~\ref{thm:fgalgebra} are satisfied with $D=A/\m$, with $D$ Egyptian, hence locally Egyptian.  If on the other hand $\ker \mu = (0)$, then by Theorem~\ref{thm:prevEgyptian}(\ref{it:Z} and \ref{it:alg}), the hypotheses of Theorem~\ref{thm:fgalgebra} are satisfied with $D=A$, with $D$ Egyptian, hence locally Egyptian.
\end{proof}

\section{Pullbacks and integrality}
In this section, we concentrate on a situation where an integral domain $S$ shares a nonzero ideal $I$ with a subring $R$ of $S$.  This is classically known as a \emph{pullback}, and includes many common constructions such as the $D+M$ construction. It is shown here that if $S$ is Egyptian (resp. generically Egyptian), then so is $R$ .

We start with a lemma that will be useful in dealing with the shared ideal.
\begin{lemma}\label{lem:egyptianideal}
Let $R$ be an Egyptian domain, and let $I$ be a nonzero ideal of $R$.  Then for any nonzero $x \in R$, there is a representation of $x$ in the form \[
x = \frac 1 {i_1} + \cdots + \frac 1 {i_n},
\]
where $i_1, \ldots, i_n$ are distinct nonzero elements of the ideal $I$.
\end{lemma}

\begin{proof}
Choose an element $0\neq i \in I$.  There are distinct elements $r_1, \ldots, r_n$ of $R$ such that \[
xi = \frac 1 {r_1} + \cdots + \frac 1 {r_n}.
\]
Dividing through by $i$, we have $x = \sum_{j=1}^n \frac 1 {ir_j}$, and all the $ir_j$ are distinct from one another since $i(r_j - r_k) \neq 0$ whenever $j \neq k$ (since $R$ is an integral domain and $i, r_j - r_k \neq 0$).  Then setting $i_j := ir_j$, we are done.
\end{proof}

\begin{thm}\label{thm:pullE}
Let $S$ be an Egyptian (resp. a generically Egyptian) domain, $R$ a subring of $S$, and $I$ a nonzero shared ideal of $R$ and $S$.  Then $R$ is Egyptian (resp. generically Egyptian).
\end{thm}

\begin{proof}
First suppose $S$ is Egyptian.  Let $0\neq r \in R$.  Then since $r\in S$ and $I$ is a nonzero ideal of $S$, by Lemma~\ref{lem:egyptianideal} there are nonzero $i_1, \ldots, i_n \in I$ such that \[
r = \frac{1}{i_1} + \cdots + \frac{1}{i_n}.
\]
But since the $i_j$ are in $R$, it follows that $r$ is an Egyptian element of $R$.  Thus, $R$ is an Egyptian domain.

Now suppose $S$ is generically Egyptian. Let $0\neq x \in S$ such that $S_x$ is Egyptian.  Let $0 \neq i \in I$. Then $R_{ix} = S_{ix}$ as subrings of their common fraction field, and is hence an overring of the Egyptian domain $S_x$.  By Theorem~\ref{thm:prevEgyptian}(\ref{it:overring}), it follows that $R_{ix}$ is Egyptian, and hence $R$ is generically Egyptian.
\end{proof}

\begin{cor}\label{cor:mf}
Let $R \subseteq S$ be integral domains with the same fraction field, such that $S$ is finitely generated as a module over $R$ (e.g. if $R$ is a quasi-excellent Noetherian domain and $S$ is its integral closure, or if $R$ is any integral domain and $S=R[\alpha]$, where $\alpha$ is an element of the fraction field that is integral over $R$).  Then $R$ is Egyptian (resp. generically Egyptian) if and only if $S$ is.
\end{cor}

\begin{proof}
If $R$ is Egyptian (resp. generically Egyptian), then $S$ must be as well by Theorem~\ref{thm:prevEgyptian}(\ref{it:overring}) (resp. Lemma~\ref{lem:overring}).
For the converse, suppose $S$ is Egyptian (resp. generically Egyptian).  The conditions imply that the \emph{conductor} $(R :_R S)$ is a common nonzero ideal to $R$ and $S$ \cite[Exercise 2.11]{HuSw-book}.  Then by Theorem~\ref{thm:pullE}, $R$ is Egyptian (resp. generically Egyptian).
\end{proof}
\section{Questions}

\begin{question}
Can one find a Noetherian domain that is not generically (or locally) Egyptian?  
\end{question}

\begin{question}
In Example~\ref{ex:FXlocally}, we saw that  a domain can be locally Egyptian without being Egyptian. But can a domain be generically Egyptian without being locally Egyptian?  
\end{question}

\begin{question}\label{q:pullback}
Let $S$ be an integral domain, let $R$ be a subring of $S$, and suppose $I$ is a nonzero common ideal to $R$ and $S$.
If $S$ is locally Egyptian, then is $R$ locally Egyptian?
\end{question}

\begin{rem}\label{rem:overringpb}
Note that the converse of Question~\ref{q:pullback} has a positive answer.  This is because under these conditions, $S$ is always an overring of $R$, so if $c_1,\ldots, c_n \in R$ generate the unit ideal and each $R_{c_i}$ is Egyptian, then $S_{c_i}$ is an overring of $R_{c_i}$, so is Egyptian by Theorem~\ref{thm:prevEgyptian}(\ref{it:overring}), whence $S$ is locally Egyptian.  The same implication holds for ``Egyptian'' and ``generically Egyptian''.
\end{rem}

\begin{question}
Let $R \subseteq S$ be a domain extension, such that $S$ is integral over $R$.  We know that if $R$ is Egyptian, so is $S$ by Theorem~\ref{thm:prevEgyptian}(\ref{it:alg}), since in particular $S$ is algebraic over $R$.  (Indeed, the proof in \cite{GLO-Egypt} is really about integrality, as the key step involves a monic polynomial.)  But does the converse hold?

This question is even interesting when $R$ and $S$ share a fraction field.  A restatement would be like this: ``Let $S$ be an integrally closed Egyptian domain and let $R$ be a subring whose integral closure is $S$.  Is $R$ Egyptian?''  It's a restatement since overrings of Egyptian domains are always Egyptian by Theorem~\ref{thm:prevEgyptian}(\ref{it:overring}).  If it were true, it would ``reduce'' the study of Egyptianness to integrally closed domains.  By Corollary~\ref{cor:mf}, the latter question has a positive answer for most rings that typically come up in Noetherian and algebraic-geometric settings.
\end{question}

\section*{Acknowledgments}
Thanks to the authors of \cite{GLO-Egypt} for providing me with an advance copy of their paper.  Similarly, I am grateful for interesting e-mail discussions with Greg Oman and Lorenzo Guerrieri, following Oman's interesting talk at the AMS sectional meeting in Chattanooga.  In particular, Example~\ref{ex:FXlocally} is due to Oman, and Example~\ref{ex:infinitepoly} is due to Guerrieri.  I also found an e-mail discussion with Winfried Bruns useful, in which we determined a published source for the finitely generated version of Lemma~\ref{lem:groupcone} as an exercise in his joint book.
Finally, I would like to thank the referee for some useful corrections and clarifications that have improved the paper.

\providecommand{\bysame}{\leavevmode\hbox to3em{\hrulefill}\thinspace}
\providecommand{\MR}{\relax\ifhmode\unskip\space\fi MR }
\providecommand{\MRhref}[2]{%
  \href{http://www.ams.org/mathscinet-getitem?mr=#1}{#2}
}
\providecommand{\href}[2]{#2}

\end{document}